\documentclass[12pt, reqno]{amsart}
\usepackage[utf8]{inputenc}
\usepackage{amsmath}
\usepackage{amsfonts}
\usepackage{amssymb}
\usepackage{amsthm}
\usepackage{bm}
\usepackage{bbm}
\usepackage{xcolor}
\usepackage{url}
\usepackage{hyperref}

\numberwithin{equation}{section}

\theoremstyle{plain}
\newtheorem{theorem}{Theorem}[section]
\newtheorem{lemma}[theorem]{Lemma}
\newtheorem{corollary}[theorem]{Corollary}
\newtheorem{proposition}[theorem]{Proposition}
\newtheorem{definition}[theorem]{Definition}

\title[Cubes from products with one term missing]{Cubes from products of terms in progression with one term missing}
\author{Kyle Pratt}
\address{Brigham Young University, Department of Mathematics, Provo, UT 84602, USA}
\email{kyle.pratt@mathematics.byu.edu}

\subjclass[2020]{11D41, 11G05, 11Y50}
\keywords{Diophantine equations, products of terms in arithmetic progression, elliptic curves, elliptic curve Chabauty}

\allowdisplaybreaks

\begin{document}
\date{}

\maketitle

\begin{abstract}
Let $5 \leq k \leq 11$ and $0\leq i \leq k-1$ be integers. We determine all solutions to the equation
\begin{align*}
n(n+d)(n+2d)\cdots(n+(i-1)d)(n+(i+1)d) \cdots (n+(k-1)d) = y^3
\end{align*}
in integers $n,d,y$ with $ny \neq 0$, $d\geq 1$, and $\textup{gcd}(n,d) = 1$. Our method relies on the theory of elliptic curves, including elliptic curve Chabauty over a number field. As an application, we answer a question of Das, Laishram, Saradha, and Sharma concerning rational points on a certain superelliptic curve.
\end{abstract}

\section{Introduction}

Erd\H{o}s and Selfridge famously proved that the product of two or more positive consecutive integers is never a perfect power \cite{ErdSel1975}. That is, they showed the equation
\begin{align}\label{eq:ErdSelf eqn}
n(n+1)(n+2)\cdots(n+k-1) = y^\ell
\end{align}
has no solutions in positive integers $n,y,k,\ell$ with $k,\ell \geq 2$. 

There are many ways to generalize the Diophantine equation \eqref{eq:ErdSelf eqn}. One generalization, which has been studied a great deal, is as follows: determine whether there are solutions to
\begin{align}\label{eq:generalized ErdSelf eqn}
n(n+d)(n+2d) \cdots (n+(k-1)d) = y^\ell
\end{align}
in positive integers $n,d,y,k,\ell$ with $\text{gcd}(n,d)=1$ and $k,\ell \geq 2$. Hence, rather than taking products of consecutive integers as in \eqref{eq:ErdSelf eqn}, one takes products of consecutive terms in an arithmetic progression. (See \cite[pp. 22--26]{Sho2006} for a survey of some older results on \eqref{eq:generalized ErdSelf eqn} and its variants. See \cite{Ben2018,BenSik2016,BenSik2020,DLS2018,DLSS2023,GHP2009,HK2011,HTT2009}, say, and the associated references for an introduction to more recent work on this and related problems.) It is widely believed that \eqref{eq:generalized ErdSelf eqn} has no solutions once $k$ is sufficiently large (see, e.g., \cite[p. 356]{BenSik2020} and \cite{ErdProb672}). In fact, it ought to be the case that one can remove some terms $n+jd$ from the left-hand side of \eqref{eq:generalized ErdSelf eqn}, and the resulting equation will still be unsolvable (for $k$ sufficiently large). 

Some recent work in this direction is due to Bennett \cite{Ben2023}, who studied equations of the form
\begin{align}\label{eq:Bennett eqn}
\prod_{\substack{j=0 \\ j \neq i}}^{k-1} (n+jd) = y^\ell,
\end{align}
where $5 \leq k \leq 8, i \in \{0,1,\ldots,k-1\}$, and $\ell \geq 3$ is a prime. Thus, these are equations where a single term $n+id$ is removed from the product. A variety of methods are used in \cite{Ben2023} to study this equation, such as the modularity of elliptic curves and a clever analysis of the resulting congruences (for $\ell \geq 7$), rational points on elliptic curves of rank zero (for $\ell=3$), and Chabauty-type methods (for $\ell=5$, relying on work of Hajdu and Kov\'{a}cs \cite{HK2011}). It is claimed in \cite[Theorem 1.2]{Ben2023} that all solutions to \eqref{eq:Bennett eqn} with $5 \leq k \leq 8$ arise when $k=5$ and $\ell=3$.  However, the analysis in \cite{Ben2023} is incomplete in the case $\ell=3$, as some solutions to \eqref{eq:Bennett eqn} are missing.

The main result of this paper is the following theorem.

\begin{theorem}\label{thm:main thm}
Let $5 \leq k \leq 11$ and $0\leq i \leq k-1$ be integers. The only solutions to the equation
\begin{align}\label{eq:main eqn of study}
\prod_{\substack{j=0 \\ j \neq i}}^{k-1}(n+jd) = y^3
\end{align}
in integers $n,d,y$ with $ny \neq 0$, $d \geq 1$, and $\textup{gcd}(n,d) = 1$ are given by
\begin{equation}\label{eq:solns for main thm}
\begin{split}
(k,i,n,d) = \, &(5,0,-14,5),(5,0,-11,5),(5,1,-8,3),(5,3,-4,3), \\
&(5,4,-9,5),(5,4,-6,5), (7,3,-10,7),(7,3,-32,7).
\end{split}
\end{equation}
\end{theorem}

It would be interesting to extend Theorem \ref{thm:main thm} to larger values of $k$, and to allow for more terms to be missing from the products.

Theorem \ref{thm:main thm} allows us to answer a question raised by Das, Laishram, Saradha, and Sharma \cite[p. 1710]{DLSS2023} concerning rational points on a certain superelliptic curve.

\begin{corollary}\label{cor:rat pts}
All rational points $(x,y) \in \mathbb{Q}^2$ on the curve 
\begin{align}\label{eq:dlss curve}
y^3=(x+1)(x+2)(x+3)(x+5)(x+6)(x+7)
\end{align}
satisfy $y=0$, or $(x,y) = \left(-\frac{17}{7},\frac{120}{49} \right)$, or $(x,y) = \left(-\frac{39}{7},\frac{120}{49} \right)$.
\end{corollary}

The points on \eqref{eq:dlss curve} with $y=0$ are the ``trivial'' rational points
\begin{align*}
(-1,0),(-2,0),(-3,0),(-5,0),(-6,0),(-7,0).
\end{align*}
The authors of \cite{DLSS2023} noted the existence of the ``nontrivial'' rational point $\left(-\frac{17}{7},\frac{120}{49} \right)$, but did not provably determine all of the rational points on \eqref{eq:dlss curve}.

\subsection{Outline of the paper} 

In Section \ref{sec:prelim}, we show how Corollary \ref{cor:rat pts} follows from Theorem \ref{thm:main thm}. We then prove an elementary but important lemma (Lemma \ref{lem:factorization of terms}) detailing the structure of the terms $n+jd$ in any solution to \eqref{eq:main eqn of study}. This gives rise to the notion of a \emph{coefficient vector}  (Definition \ref{defn: coeff vector}). A large part of our work is devoted to determining the coefficient vectors that can correspond to solutions to \eqref{eq:main eqn of study}. It will transpire that solutions to \eqref{eq:main eqn of study} give rise to many different ternary cubic equations (see \eqref{eq:ternary cubic eqn}), and these cubic equations can be shown to be unsolvable in many cases by relating them to elliptic curves of rank zero (Lemma \ref{lem:selmer}). We end the section with a general description of our strategy for determining potential coefficient vectors, which relies extensively on computer calculation. (All the computer code used in this paper, written in Magma \cite{Magma} and Sage \cite{Sage}, is available at our ``Cubes from products with one term missing'' GitHub repository \cite{Pratt_cubes_2025}.)

In Section \ref{sec:proving most of main thm}, we use elementary arguments and the results of Section \ref{sec:prelim} to solve \eqref{eq:main eqn of study} in all cases of $k$ and $i$ except $k=7,i=3$. There are two coefficient vectors in this case that correspond to solutions to \eqref{eq:main eqn of study} (see \eqref{eq:solns for main thm}). All the associated ternary cubic equations correspond to elliptic curves of positive rank, so the results of Section \ref{sec:prelim} are inapplicable.

In the last section of the paper, Section \ref{sec:finishing proof of main thm}, we finish the analysis in the case $k=7,i=3$ and give the proof of Theorem \ref{thm:main thm}, using a result on the simultaneous solutions to a pair of ternary cubic equations (Proposition \ref{prop:only solns to pair of cubic eqns}). Computations in the pure cubic field $K = \mathbb{Q}(\sqrt[3]{5})$ show that any solution to the pair of cubic equations gives rise to a point on an elliptic curve $E$ over $K$. All the points on $E$ arising in this way can be determined with the elliptic curve Chabauty method (see Lemma \ref{lem:EC over K}).

\section{Preliminaries and proof strategy}\label{sec:prelim}

We first show how Corollary \ref{cor:rat pts} follows from Theorem \ref{thm:main thm}.
\begin{proof}[Proof of Corollary \ref{cor:rat pts} assuming Theorem \ref{thm:main thm}]
Let $(x,y) \in \mathbb{Q}^2$ satisfy
\begin{align*}
y^3=(x+1)(x+2)(x+3)(x+5)(x+6)(x+7).
\end{align*}
We may assume $y \neq 0$. Since $1\cdot 2 \cdot 3 \cdot 5 \cdot 6 \cdot 7$ is not a perfect cube, we may also assume $x \neq 0$.  We then write $y = t/u$ and $x=a/b$, where $a$ and $t$ are nonzero integers, $u$ and $b$ are positive integers, and $\text{gcd}(t,u)=\text{gcd}(a,b)=1$. Inserting the expressions for $x,y$ and multiplying through by $b^6u^3$, we obtain
\begin{align*}
t^3b^6 = u^3(a+b)(a+2b)(a+3b)(a+5b)(a+6b)(a+7b).
\end{align*}
This implies $u^3 \mid t^3b^6$, but since $\text{gcd}(t,u)=1$ we have $u^3 \mid b^6$. We similarly find $b^6 \mid u^3$, and since $b,u$ are both positive we deduce $b^6=u^3$ (thus $u=b^2$). Canceling, we find
\begin{align*}
t^3 = (a+b)(a+2b)(a+3b)(a+5b)(a+6b)(a+7b).
\end{align*}
After a suitable change of variables, we have a solution to \eqref{eq:main eqn of study} with $k=7,i=3$. Since $x = a/b, y = t/b^2$, we finish by appealing to Theorem \ref{thm:main thm}.
\end{proof}

We state some definitions in preparation for further discussion. An integer $n$ is \emph{cube-free} if $n$ is not divisible by $p^3$ for any prime $p$. Any nonzero integer $n$ can be written uniquely as $n=cf^3$, where $c$ is positive and cube-free; we say $c$ is the \emph{cube-free part} of $n$ (note that the cube-free part of a nonzero integer $n$ is always positive, regardless of the sign of $n$).

As with all works studying solutions to equations like \eqref{eq:main eqn of study}, we make use of elementary observations regarding factorizations of the terms $n+jd$.

\begin{lemma}\label{lem:factorization of terms}
Let $5 \leq k \leq 11$ and $0 \leq i \leq k-1$ be integers. Assume there are nonzero integers $y,n,d$ with $d\geq 1$ and $\textup{gcd}(n,d)=1$ such that \eqref{eq:main eqn of study} holds. Then for $0 \leq j \leq k-1, j \neq i$, we may uniquely write $n+jd=a_jx_j^3$, where:
\begin{itemize}
\item $a_j$ is a positive integer,
\item $a_j$ is cube-free, and divisible only by primes $\leq k-1$,
\item $\textup{gcd}(a_j,a_\ell) \mid (\ell-j)$ if $0 \leq j < \ell \leq k-1$ and $j,\ell \neq i$,
\item the integers $x_j$ are nonzero, and if $k \leq 8$ they are pairwise coprime,
\item $\prod_{\substack{j=0 \\ j \neq i}}^{k-1}a_j$ is a perfect cube.
\end{itemize}
\end{lemma}
\begin{proof}
Let $0 \leq j < \ell \leq k-1$ with $j,\ell \neq i$, and let $g = \text{gcd}(n+jd,n+\ell d)$. Then $g$ divides $(n+\ell d)-(n+jd) = (\ell-j)d$, and since $\text{gcd}(n,d)=1$ we find $g$ is coprime to $d$. Hence $g \mid (\ell - j)$. It follows that the greatest common divisor of $n+jd,n+\ell d$ is at most $k-1$. Therefore, if $p\geq k$ is a prime and $p \mid y$, then $p$ divides a unique term $n+jd$. 

We may therefore uniquely factor
\begin{align*}
n+jd = e_jz_j^3,
\end{align*}
where $e_j$ is only divisible by primes $\leq k-1$, and $z_j$ is only divisible by primes $\geq k$. We may assume $e_j$ is a positive integer by pulling a minus sign into $z_j^3$.  From the discussion above, we have $\text{gcd}(z_j,z_\ell) = 1$ if $j \neq \ell$. 

By pulling cube factors of $e_j$ into $z_j$, we may then write
\begin{align*}
n+jd = a_jx_j^3,
\end{align*}
where $a_j$ is positive, cube-free, and divisible only by primes $\leq k-1$. Note that $x_j \neq 0$ since $y\neq 0$. Note also that if there is some prime $p$ with $p \mid \text{gcd}(x_j,x_\ell)$, then $p^3 \mid \text{gcd}(n+jd,n+\ell d)$, so we must have $8\leq p^3 \leq k-1$, or $k \geq 9$.

Lastly, we note that
\begin{align*}
y^3 = \prod_{\substack{j=0 \\ j \neq i}}^{k-1} (n+jd) = \prod_{\substack{j=0 \\ j \neq i}}^{k-1} a_j\cdot \Big(\prod_{\substack{j=0 \\ j \neq i}}^{k-1} x_j \Big)^3,
\end{align*}
so the product of the $a_j$'s is a perfect cube.
\end{proof}

The factorizations described in Lemma \ref{lem:factorization of terms} are of fundamental importance. We pay special attention to the integers $a_j$.

\begin{definition}\label{defn: coeff vector}
Assume there are nonzero integers $y,n,d$ with $d\geq 1$ and $\textup{gcd}(n,d)=1$ such that \eqref{eq:main eqn of study} holds. Let $n+jd=a_jx_j^3$ be the factorization as described in Lemma \ref{lem:factorization of terms}. We call $(a_0,a_1,\ldots,a_{i-1},a_{i+1},\ldots,a_{k-1})$ the \emph{coefficient vector} corresponding to $y,n,d$.
\end{definition}

Solutions to \eqref{eq:main eqn of study} correspond to their coefficient vectors, and, given a potential coefficient vector, we can ask whether there are any corresponding solutions to \eqref{eq:main eqn of study}. We will sometimes abbreviate ``coefficient vector'' to just ``vector.''

Our general strategy is to determine the coefficient vectors that can correspond to solutions to \eqref{eq:main eqn of study}. Since the integers $a_j$ are cube-free and their largest prime factors are bounded, there are only a finite number of potential coefficient vectors to consider for given values of $k$ and $i$. Given a potential coefficient vector $\textbf{v} = (a_0,\ldots,a_{i-1},a_{i+1},\ldots,a_{k-1})$, we may try to show $\textbf{v}$ cannot correspond to solutions of \eqref{eq:main eqn of study} by reducing to various cubic equations. More particularly, given distinct integers $r,s,t$ with $0 \leq r,s,t \leq k-1$ and $r,s,t \neq i$, we have
\begin{equation}\label{eq:ternary cubic eqn}
\begin{split}
0 &= (s-t)(n+rd)+(t-r)(n+sd)+(r-s)(n+td) \\
&= (s-t)a_rx_r^3 + (t-r)a_sx_s^3 + (r-s)a_tx_t^3.
\end{split}
\end{equation}
If the cubic equation $(s-t)a_rx_r^3 + (t-r)a_sx_s^3 + (r-s)a_tx_t^3 = 0$ has no solutions in nonzero integers $x_r,x_s,x_t$, then $\textbf{v}$ cannot be a coefficient vector corresponding to any solutions to \eqref{eq:main eqn of study}.

In order for this strategy to be effective, of course, we need access to a large supply of cubic equations of the form $ax^3+by^3+cz^3=0$ that we know to be unsolvable. The following lemma, which is essentially a variant of \cite[Proposition 4.2]{Ben2023}, is our main tool.

\begin{lemma}\label{lem:selmer}
Let
\begin{align*}
\mathcal{L} = \{&1, 3, 4, 5, 10, 14, 18, 21, 25, 36, 45, 60, 100, 147, 150, \\
                &175, 196, 225, 245, 252, 300, 315, 350, 882, 980, \\
                &1050, 1470, 1575, 1764, 2940, 7350, 14700\}.
\end{align*}
Let $a,b,c$ be nonzero integers, and let $D$ denote the cube-free part of $abc$. If $D \in \mathcal{L}$, then the equation $ax^3+by^3+cz^3=0$ has no solutions in nonzero integers $x,y,z$.
\end{lemma}
\begin{proof}
We may suppose that $a,b,c$ are all positive integers by pulling any minus signs into $x,y,z$. If $ax^3+by^3+cz^3=0$ has a solution in nonzero integers $x,y,z$, then, by \cite[Theorem I and equation 1.2.4]{Sel1951}, the equation $X^3+Y^3+abcZ^3=0$ has a solution in integers $X,Y,Z$ with $Z \neq 0$. Absorbing cube factors of $abc$ and its sign into $Z^3$, we see that $X^3+Y^3+DU^3=0$ has a solution in integers with $U \neq 0$.

We must have $X+Y \neq 0$ since $U \neq 0$. If we set
\begin{align*}
F = -12D \frac{U}{X+Y}, \ \ \ \ G = 36D \frac{X-Y}{X+Y},
\end{align*}
then $G^2=F^3-432D^2$. In other words, $(F,G)$ is a rational point on the elliptic curve
\begin{align*}
E_D : v^2=u^3-432D^2.
\end{align*}
If $D\geq 3$ and $D$ is cube-free, then the torsion subgroup of $E_D(\mathbb{Q})$ is trivial (see \cite[exercise 10.19]{AEC2009}). For those $D \in \mathcal{L}$ with $D \geq 3$, we find\footnote{See \text{Lemma-2d3-elliptic-curve-computations.sage} in our GitHub repository \cite{Pratt_cubes_2025}.} that $E_D$ has rank zero, and therefore $E_D(\mathbb{Q})$ is empty. It follows that the original equation $ax^3+by^3+cz^3=0$ has no solutions in nonzero integers $x,y,z$.

Suppose $D=1$, so that $abc$ is a cube. Write $abc = f^3$, with $f\geq 1$. If we let
\begin{align*}
F = -\frac{abxy}{f^2z^2}, \ \ \ \ G = -\frac{1}{2}-\frac{ab(ax^3-by^3)}{2f^3z^3},
\end{align*}
then we have $G^2+G = F^3$. Therefore $(F,G)$ is a rational point on the elliptic curve $v^2+v=u^3$. This elliptic curve has rank zero, and its two non-identity points have $u=0$. This contradicts the fact that $abcxyz \neq 0$.
\end{proof}

The next lemma is sometimes useful when a coefficient vector does give rise to solutions to \eqref{eq:main eqn of study}.

\begin{lemma}\label{lem:x3 + y3 + 2z3}
The only integer solutions to $x^3+y^3+2z^3=0$ with $xyz \neq 0$ and $x,y,z$ pairwise coprime are $(x,y,z) = \pm (1,1,-1)$.
\end{lemma}
\begin{proof}
Suppose we have nonzero, pairwise coprime integers $x,y,z$ such that $x^3+y^3+2z^3=0$. We must have $x+y \neq 0$, otherwise $z=0$. We set
\begin{align*}
F = -6\frac{z}{x+y}, \ \ \ \ G = 9\frac{x-y}{x+y},
\end{align*}
to find $G^2=F^3-27$. Hence $(F,G)$ is a rational point on the elliptic curve $v^2=u^3-27$. This elliptic curve has rank zero\footnote{See \text{Lemma-2d4-elliptic-curve-computation.sage} in our GitHub repository \cite{Pratt_cubes_2025}.}, and its one non-identity point is $(u,v)=(3,0)$. Comparing this with our change of variables, we see that $x=y$ and $x=-z$. Since $x,y,z$ are pairwise coprime, we must have $(x,y,z) = \pm (1,1,-1)$, as desired.
\end{proof}

The following result is helpful for eliminating some vectors not handled by Lemmas \ref{lem:selmer} or \ref{lem:x3 + y3 + 2z3}. It implies we cannot have three consecutive coefficients $a_j$ in Lemma \ref{lem:factorization of terms} all equal to one.

\begin{lemma}\label{lem:three consec aj equal one}
Let $0 \leq j \leq k-1$ such that $j,j+1,j+2 \neq i$. Then at least one of $a_j,a_{j+1},a_{j+2}$  is not equal to one.
\end{lemma}
\begin{proof}
Assume by way of contradiction that $a_j = a_{j+1}=a_{j+2} = 1$. Then $n+jd=x_j^3, n+(j+1)d=x_{j+1}^3, n+(j+2)d=x_{j+2}^3$. Observe that $x_j,x_{j+1},x_{j+2}$ are pairwise coprime since the greatest common divisor of any two of $n+jd,n+(j+1)d,n+(j+2)d$ is at most two. We observe that
\begin{align*}
0 &= n+jd-2(n+(j+1)d) + n+(j+2)d = x_j^3 - 2x_{j+1}^3 + x_{j+2}^3,
\end{align*}
so by Lemma \ref{lem:x3 + y3 + 2z3} we have $x_j = x_{j+1}=x_{j+2}=\epsilon$ for some $\epsilon \in \{-1,1\}$. It follows that $n+jd = n+(j+1)d = \epsilon$, and therefore $d=0$. This contradicts our assumption that $d$ is positive.
\end{proof}

Our last lemma assists in the cases $i=0$ or $i=k-1$. The result is due to Hajdu, Tengely, and Tijdeman \cite[Theorem 2.1]{HTT2009}

\begin{lemma}\label{lem:HTT}
Suppose we have a solution to the equation
\begin{align*}
\prod_{j=0}^{k-1}(n+jd) = y^3
\end{align*}
in integers $k,n,d,y$ with $4 \leq k \leq 38, d \geq 1, ny \neq 0$, and $\textup{gcd}(n,d)=1$. Then $(k,n,d) = (4,-9,5)$ or $(4,-6,5)$.
\end{lemma}

The strategy is now clear. Given a potential coefficient vector $\textbf{v}$, go through all ${{k-1} \choose 3}$ cubic equations arising from choices of $r,s,t$ in \eqref{eq:ternary cubic eqn}. For each cubic equation $ax^3+by^3+cz^3=0$, we check whether an application of Lemma \ref{lem:selmer} gives rise to a contradiction. If so, then $\textbf{v}$ cannot be a coefficient vector corresponding to a solution of \eqref{eq:main eqn of study}. In practice, there are only a small number of remaining coefficient vectors to consider. We can then attempt to analyze these remaining vectors by a mix of Lemma \ref{lem:x3 + y3 + 2z3}, Lemma \ref{lem:three consec aj equal one}, and elementary arguments.

The number of possible coefficient vectors is finite for any fixed $k$ and $i$, but there are many different possible vectors. Analyzing the different possibilities by hand leads to many different cases (see, for instance, the arguments in \cite[Section 8]{DLSS2023}). In order to avoid manual case-by-case analysis, we rely instead on computer calculation. We build a list of possible coefficient vectors one entry at a time. That is, for fixed $k$ and $i$, we take incomplete coefficient vectors $(a_0,\ldots,a_{i-1},a_{i+1},\ldots,a_{j-1})$, and examine all the possible ways to extend to vectors $(a_0,\ldots,a_{i-1},a_{i+1},\ldots,a_{j-1},a_j)$. The entries in the coefficient vectors must be compatible with the conditions $\text{gcd}(n+jd,n+\ell d) \mid (\ell-j)$. We follow this process repeatedly to build up the set of possible incomplete coefficient vectors. Further, there is a unique element by which to extend an incomplete vector to its last element, since the product of the $a_j$ in a coefficient vector must be a perfect cube by Lemma \ref{lem:factorization of terms}.

When $k$ is small, there are not many choices for incomplete coefficient vectors, and the calculations finish quickly. As $k$ increases, the number of choices for incomplete coefficient vectors grows exponentially. For instance, when $k=11,i=4$, we must consider nearly fourteen million incomplete vectors. However, we are aided significantly by the (empirically-observed) fact that most almost-complete coefficient vectors cannot be extended to a full coefficient vector in a way that is consistent with various GCD constraints. For instance, when considering the case $k=11, i=4$, there are only about four hundred thousand complete coefficient vectors $(a_0,a_1,a_2,a_3,a_5,\ldots,a_{10})$ that satisfy all the necessary GCD conditions. Almost all of these vectors can then be handled by searching through ternary equations and using Lemma \ref{lem:selmer}.

\section{Proving ``most'' of Theorem \ref{thm:main thm}}\label{sec:proving most of main thm}

In this section, we prove ``most'' of Theorem \ref{thm:main thm}. It will transpire that the results in the previous section allow us to solve \eqref{eq:main eqn of study} fairly easily in all cases, except the case $k=7$ and $i=3$. We then study the case $k=7,i=3$ in the last section of the paper.

We consider several cases, depending on the values of $k$ and $i$\footnote{Sage code to verify the computations in this section is provided in \text{Section-3-computations.sage} in our GitHub repository \cite{Pratt_cubes_2025}. Information for particular values of $k$ and $i$ can be found in the ``output'' folder of the repository.}.

\subsection{$k=5$, and $i=0$ or $i=4$}

If $k=5$ and $i=0$, then we wish to solve the equation $(n+d)(n+2d)(n+3d)(n+4d)=y^3$ with $ny \neq 0, d \geq 1$, and $\text{gcd}(n,d)=1$. We set $m=n+d$ to obtain $m(m+d)(m+2d)(m+3d) = y^3$, where $m \neq 0$ and $\text{gcd}(m,d)=1$. By Lemma \ref{lem:HTT}, the only solutions to this equation have $d=5$ and $m \in \{-9,-6\}$. This solves \eqref{eq:main eqn of study} in the case $k=5,i=0$. The case with $k=5$ and $i=4$ is similar.

\subsection{$k=5, i=1$}

There are only three possible coefficient vectors not eliminated by searching through ternary equations and using Lemma \ref{lem:selmer}:
\begin{align*}
(a_0,a_2,a_3,a_4) = (1,1,1,1),(6,1,9,4),(1,2,1,4).
\end{align*}
The vector $(1,1,1,1)$ can be eliminated by using Lemma \ref{lem:three consec aj equal one} with $j=2$.

The vector $(6,1,9,4)$ can be eliminated by an elementary argument, as follows. The equations corresponding to this coefficient vector are
\begin{align*}
n &= 6x_0^3, \ \ \ \ n+2d = x_2^3, \ \ \ \ n+3d = 9x_3^3, \ \ \ \ n+4d = 4x_4^3.
\end{align*}
By Lemma \ref{lem:factorization of terms}, the integers $x_j$ are pairwise coprime. However, observe that $4 \mid n$, from the last equation, so that $2 \mid x_0$, by the first equation. The second equation implies $2 \mid x_2$, and therefore $2 \mid \text{gcd}(x_0,x_2)$, but this is a contradiction.

The vector $(1,2,1,4)$ gives rise to a solution. We choose $r=0,s=2,t=4$ in \eqref{eq:ternary cubic eqn} to obtain $x_0^3-4x_2^3+4x_4^3 = 0$, where $x_0,x_2,x_4$ are nonzero and pairwise coprime. This equation clearly implies $x_0$ is even, so write $x_0 = 2y_0$. Substituting and dividing through, we then have $2y_0^3+(-x_2)^3+x_4^3=0$. Lemma \ref{lem:x3 + y3 + 2z3} implies $(x_0,x_2,x_4) = \pm (2,1,-1)$. For some $\epsilon \in \{-1,1\}$, we then have $n=(2\epsilon)^3=8\epsilon$ and $n+2d = 2(\epsilon)^3 = 2\epsilon$. It follows that $2d = -6\epsilon$, and since $d$ is positive we must have $\epsilon = -1$. It follows that $n=-8$ and $d=3$, as in \eqref{eq:solns for main thm}.

\subsection{$k=5,i=2$}

The only possible coefficient vectors not eliminated by reducing to ternary equations and using Lemma \ref{lem:selmer} are
\begin{align*}
(a_0,a_1,a_3,a_4) = (1,1,1,1),(1,36,2,3),(3,2,36,1).
\end{align*}
We eliminate each of these vectors, following the argument of Bennett \cite[p. 219]{Ben2023}. We use the fact that, for any real $x$, we have
\begin{align}\label{eq:bennett special identity}
(x+1)^2(x+4)-x(x+3)^2 = 4.
\end{align}

Suppose $(a_0,a_1,a_3,a_4) = (1,1,1,1)$, and let $x = n/d$ in \eqref{eq:bennett special identity}. Multiplying through by $d^3$, we obtain
\begin{align*}
4d^3 = (n+d)^2(n+4d) - n(n+3d)^2 = x_1^6x_4^3 - x_0^3x_3^6,
\end{align*}
so $(x_1^2x_4)^3+(-x_0x_3^2)^3 + 4(-d)^3 = 0$. The equation $X^3+Y^3+4Z^3 = 0$ does not have solutions with $Z \neq 0$ by Lemma \ref{lem:selmer}, so $(1,1,1,1)$ cannot correspond to a solution.

Now assume $(a_0,a_1,a_3,a_4) = (1,36,2,3)$. Arguing as above, we arrive at the equation $4d^3 = 36^2 \cdot 3 \cdot x_1^6 x_4^3 - 2^2 \cdot x_0^3x_3^6$. We set $X = -d,Y=-x_0x_3^2$, and $Z=3x_1^2x_4$ to see that $4X^3+4Y^3+144Z^3=0$. Dividing through by $4$ gives $X^3+Y^3+36Z^3=0$, but this equation has no solutions in integers with $Z \neq 0$, by the proof of Lemma \ref{lem:selmer}. The last vector $(3,2,36,1)$ is eliminated similarly.

\subsection{$k=5,i=3$}

This case is similar to the case $k=5,i=1$. The three vectors to consider are: $(1,1,1,1),(4,9,1,6),(4,1,2,1)$. The first one is eliminated by Lemma \ref{lem:three consec aj equal one}. The second is eliminated by an elementary argument showing $2 \mid \text{gcd}(x_2,x_4)$. The third gives rise to a solution to \eqref{eq:main eqn of study} with $n=-4,d=3$; we handle this case with Lemma \ref{lem:x3 + y3 + 2z3}.

\subsection{Some preparatory notes for $6 \leq k \leq 11$}

We can take advantage of symmetries (already implicit above) and prior work on \eqref{eq:generalized ErdSelf eqn} to reduce the number of cases for larger values of $k$.

We first note that there are no solutions with $i=0$ or $i=k-1$. If we suppose we have a solution in this situation, then, possibly after changing variables, we have a solution to the equation
\begin{align*}
m(m+d)\cdots (m+(k'-1)d) = y^3
\end{align*}
with $my \neq 0, \text{gcd}(m,d)=1$, and $5 \leq k' \leq 10$. This equation has no solutions by Lemma \ref{lem:HTT}.

Hence, we may assume $1 \leq i \leq k-2$. Suppose we have a solution to \eqref{eq:main eqn of study}. If we set $m=n+(k-1)d$, then
\begin{align*}
y^3 &= (-1)^{k-1}\prod_{\substack{j=0 \\ j \neq i}}^{k-1}(-m+(k-1-j)d) = (-1)^{k-1} \prod_{\substack{\ell = 0 \\ \ell \neq k-1-i}}^{k-1} (-m+\ell d).
\end{align*}
If we absorb the sign $(-1)^{k-1}$ into $y^3$, then this is an equation of the form \eqref{eq:main eqn of study} with $i$ replaced by $k-1-i$. In fact, this yields a bijective involution between different solutions to \eqref{eq:main eqn of study} given by
\begin{align}\label{eq:flipping bijection}
(k,i,n,d) \leftrightarrow (k,k-1-i,-n-(k-1)d,d).
\end{align}
It therefore suffices to consider $1 \leq i \leq \frac{k-1}{2}$ in what follows.

\subsection{$k=6$}

When $i=1$, the only remaining vectors are 
\begin{align*}
(a_0,a_2,a_3,a_4,a_5)=(1,1,1,1,1),(50, 36, 1, 1, 15).
\end{align*}
The first can be eliminated by Lemma \ref{lem:three consec aj equal one}. For the second coefficient vector, note that we have $n=50x_0^3$ and $n+4d = x_4^3$. Since $n$ is even, we see that $n+4d$ is even, so $x_4$ is even, and $x_4^3$ is divisible by eight. It follows that $4 \mid n$. Since $n=50x_0^3$, we must have $2 \mid x_0$. Hence $2 \mid \text{gcd}(x_0,x_4)$, and this contradicts Lemma \ref{lem:factorization of terms}.

When $i=2$, the only remaining vector is $(a_0,a_1,a_3,a_4,a_5)=(1,1,1,1,1)$, and again we use Lemma \ref{lem:three consec aj equal one}.

\subsection{$k=7$}

For $1\leq i \leq 3$, the vector $(1,1,1,1,1,1)$ remains, but it can be eliminated by Lemma \ref{lem:three consec aj equal one}. This is the only vector for $i=1$ or $i=2$. For $i=3$, however, we have two other vectors
\begin{align}\label{eq:remaining vectors for k7 i3}
(a_0,a_1,a_2,a_4,a_5,a_6) = (4, 25, 18, 4, 3, 10), (10, 3, 4, 18, 25, 4)
\end{align}
that are not eliminated by Lemma \ref{lem:selmer}. Handling these coefficient vectors requires additional effort, and we postpone their treatment until Section \ref{sec:finishing proof of main thm}.

\subsection{$8 \leq k \leq 11$}

For each $k$ and $1 \leq i \leq \frac{k-1}{2}$, the only possible coefficient vector not eliminated by searching through ternary equations and using Lemma \ref{lem:selmer} is the coefficient vector with every $a_j=1$. In every case, this can be eliminated by using Lemma \ref{lem:three consec aj equal one}.

\section{The proof of Theorem \ref{thm:main thm}}\label{sec:finishing proof of main thm}

In this section we finish the proof of Theorem \ref{thm:main thm} by handling the case $k=7,i=3$. If we attempt to follow the line of argument in the previous section, we encounter ternary equations whose associated elliptic curves all have positive rank. Thus, we must follow a slightly more sophisticated path.

\begin{proposition}\label{prop:only solns to pair of cubic eqns}
The only solutions $(x,y,z,w) \in \mathbb{Z}^4$ to the pair of equations
\begin{align*}
x^3+y^3&=9z^3 \\
5x^3-y^3&=3w^3
\end{align*}
with $xyzw \neq 0$ and $x,y$ coprime are $(x,y,z,w) = \pm (1,2,1,-1)$.
\end{proposition}

\begin{proof}[Proof of Theorem \ref{thm:main thm} assuming Proposition \ref{prop:only solns to pair of cubic eqns}]

The work of Section \ref{sec:proving most of main thm} proves Theorem \ref{thm:main thm} in all cases except the case with $k=7$ and $i=3$. In this case, by \eqref{eq:remaining vectors for k7 i3}, we know that the coefficient vector of any potential solution must satisfy
\begin{align*}
(a_0,a_1,a_2,a_4,a_5,a_6) = (4,25,18,4,3,10) \text{ or } (10,3,4,18,25,4).
\end{align*}
Equation \eqref{eq:flipping bijection} gives a bijection between solutions with coefficient vector equal to $(4,25,18,4,3,10)$ and solutions with coefficient vector equal to $(10,3,4,18,25,4)$. Since the bijection does not change $d$, it suffices to consider
\begin{align*}
(a_0,a_1,a_2,a_4,a_5,a_6) = (10,3,4,18,25,4).
\end{align*}

We apply \eqref{eq:ternary cubic eqn} with $(r,s,t)=(2,6,4)$ and $(2,6,1)$. Taking $(r,s,t) = (2,6,4)$ gives
\begin{align*}
0 &= (6-4)a_2x_2^3 + (4-2)a_6x_6^3+(2-6)a_4x_5^3 = 8x_2^3+8x_6^3-72x_4^3.
\end{align*}
Dividing through by $8$ and rearranging gives $x_2^3+x_6^3=9x_4^3$. Similarly, taking $(r,s,t) = (2,6,1)$ yields the equation $5x_2^3-x_6^3=3x_1^3$. Recall that the integers $x_j$ are pairwise coprime by Lemma \ref{lem:factorization of terms}. Therefore, $(x,y,z,w)=(x_2,x_6,x_4,x_1)$ is a solution to the pair of equations
\begin{align*}
x^3+y^3=9z^3, \ \ \ \ \ \ 5x^3-y^3=3w^3
\end{align*}
with $xyzw \neq 0$ and $x,y$ coprime. By Proposition \ref{prop:only solns to pair of cubic eqns}, it follows that $(x_2,x_6,x_4,x_1) = \pm (1,2,1,-1)$. We take $x_2 = \epsilon, x_6 = 2\epsilon$, where $\epsilon \in \{-1,1\}$. Then
\begin{align*}
n+2d &= 4\epsilon^3 = 4\epsilon, \\
n+6d &= 4(2\epsilon)^3 = 32\epsilon,
\end{align*}
which implies $d = 7\epsilon$. Since $d$ is positive, we must have $\epsilon = 1$ and $d = 7$, hence $n = -10$. The bijection \eqref{eq:flipping bijection} gives the other solution $n=-32,d=7$.
\end{proof}

We have therefore reduced matters to proving Proposition \ref{prop:only solns to pair of cubic eqns}. We need the following result.

\begin{lemma}\label{lem:EC over K}
Let $\alpha = \sqrt[3]{5}$, and set $K = \mathbb{Q}(\alpha)$. Write $\mathcal{O}_K = \mathbb{Z}[\alpha]$ for the ring of integers of $K$. The only solutions $(x,y,u) \in \mathbb{Z} \times \mathbb{Z}\times \mathcal{O}_K$ to
\begin{align}\label{eq:cubic over K}
(\alpha x - y)(x^2-xy+y^2) = 3(2-\alpha) u^3
\end{align}
with $xy \neq 0$ and $x,y$ coprime are $(x,y,u) = \pm (1,2,-1)$.
\end{lemma}
\begin{proof}
We may use the point $(x,y,u) = (1,2,-1)$ to transform \eqref{eq:cubic over K} into an elliptic curve over $K$ in Weierstrass form\footnote{Sage code to verify the Weierstrass transformation is provided in \text{Lemma-4d2-weierstrass.sage} in our GitHub repository \cite{Pratt_cubes_2025}.}. Define constants
\begin{align*}
\rho &= \frac{27}{5}\alpha^2+\frac{9}{5}\alpha-\frac{99}{5}, \\
\sigma &=  \frac{1377}{25}\alpha^2-\frac{891}{25}\alpha-\frac{2511}{25}, \\
c_1 &= -\frac{1}{5}\alpha^2 - \alpha + 1, \\
c_2 &= -\frac{5}{18}\alpha^2 - \frac{5}{9}\alpha - \frac{5}{6}, \\
c_3 &= \frac{1}{6}\alpha^2 +\frac{5}{18}\alpha + \frac{5}{9}.
\end{align*}
We write points on the elliptic curve in homogeneous coordinates $(X,Y,Z)$, and find that if
\begin{align}
(X,Y,Z) = (-u,c_1 x, c_2x+c_3 y),
\end{align}
then $(X,Y,Z)$ is a point on the elliptic curve
\begin{align*}
E : Y^2Z + \rho YZ^2 = X^3 + \sigma Z^3.
\end{align*}
The rank of $E/K$ is two. Since $2 < 3=[K:\mathbb{Q}]$, this opens the possibility of using the elliptic curve Chabauty method \cite{Bru2003}.

By considering the linear system relating $X,Y,Z$ to $x,y,u$, we find
\begin{align*}
(x,y,u) = (d_1Y, d_2Y+d_3Z,-X),
\end{align*}
where
\begin{align*}
d_1 &= -\frac{1}{6}\alpha^2 - \frac{1}{6}\alpha = \frac{1}{c_1}, \\
d_2 &= -\frac{1}{6}\alpha^2 - \frac{5}{6} = -\frac{c_2}{c_1c_3}, \\
d_3 &= -\frac{9}{5}\alpha^2 - \frac{9}{5}\alpha + 9 = \frac{1}{c_3}.
\end{align*}
It therefore suffices to determine all those points $(X,Y,Z)$ on $E$ such that $(d_1Y,d_2Y+d_3Z)$ is rational. More specifically, we study the map
\begin{align*}
\varphi: E \rightarrow \mathbb{P}^1(\mathbb{Q}), \ \ \ \ (X,Y,Z) \mapsto (d_1Y,d_2Y+d_3Z).
\end{align*}
A Magma computation using Chabauty and the Mordell-Weil sieve \cite{BS2010,Sik2015} finds that $\varphi(E) = \{(1/2,1)\}$\footnote{Magma code to verify this computation is provided in \text{Lemma-4d2-chabauty.m} in our GitHub repository \cite{Pratt_cubes_2025}.}. It follows that if $(x,y,u) \in \mathbb{Z}\times \mathbb{Z} \times \mathcal{O}_K$ is a solution to \eqref{eq:cubic over K}, then $x/y =1/2$. Since $x$ and $y$ are coprime integers, we have $x=\epsilon,y=2\epsilon$, for $\epsilon \in \{\pm 1\}$. Inserting these expressions for $x$ and $y$ into \eqref{eq:cubic over K}, we find $u =- \epsilon$.
\end{proof}

\begin{proof}[Proof of Proposition \ref{prop:only solns to pair of cubic eqns}]
Assume $x,y,z,w$ are nonzero integers with $x,y$ coprime such that $x^3+y^3=9z^3$ and $5x^3-y^3=3w^3$. Since $x$ and $y$ are coprime, we see $3 \nmid xy$; we will use this fact repeatedly in what follows.

The proof strategy is to show $x$ and $y$ must satisfy \eqref{eq:cubic over K} for some $u \in \mathcal{O}_K$, whereupon an appeal to Lemma \ref{lem:EC over K} finishes the proof. In order to carry this out, we factor $x^3+y^3$ and $5x^3-y^3$ and study potential common factors.

We begin by studying the first equation $x^3+y^3=9z^3$. We factor the sum of cubes to get 
\begin{align*}
(x+y)(x^2-xy+y^2)=9z^3.
\end{align*}
We claim that $\text{gcd}(x+y,x^2-xy+y^2)=3$. If $p$ is a prime that divides $x+y$ and $x^2-xy+y^2$, then $y\equiv -x \pmod{p}$ and $0 \equiv x^2-xy+y^2\equiv 3x^2 \pmod{p}$. Since $x$ and $y$ are coprime we have $p \nmid x$, so $p=3$. We similarly see that $3 \mid (x+y)$ implies $3 \mid (x^2-xy+y^2)$. If $3 \mid (x^2-xy+y^2)$ but $3 \nmid (x+y)$, then $0 \equiv x^2-xy+y^2 = \frac{x^3+y^3}{x+y} \equiv \frac{x+y}{x+y} \equiv 1\pmod{3}$, and this is a contradiction. Therefore, $3 \mid (x+y)$ if and only if $3 \mid (x^2-xy+y^2)$. We deduce that $x+y$ and $x^2-xy+y^2$ are each divisible by $3$ since their product is. Finally, we see that $9 \nmid (x^2-xy+y^2)$, since if we set $y=-x+3k$ for some $k \in \mathbb{Z}$, we have
\begin{align*}
x^2-xy+y^2 = 3(3k^2-3kx+x^2),
\end{align*}
and the term in parentheses is not divisible by $3$. We therefore have $\text{gcd}(x+y,x^2-xy+y^2)=3$, and since $(x+y)(x^2-xy+y^2)=9z^3$ we deduce
\begin{align}\label{eq:fac diff of cubes to get three times cubes}
x+y &= 3m^3, \ \ \ \ \ x^2-xy+y^2=3n^3
\end{align}
for some integers $m,n$.

Now we turn to the equation $5x^3-y^3=3w^3$. We factor the left-hand side over the number field $K=\mathbb{Q}(\sqrt[3]{5})=\mathbb{Q}(\alpha)$ appearing in the statement of Lemma \ref{lem:EC over K}, so that 
\begin{align}\label{eq:factored 5x3-y3}
5x^3-y^3=(\alpha x - y)(\alpha^2 x^2+\alpha x y + y^2)=3w^3.
\end{align}
Recall the ring of integers $\mathcal{O}_K$ of $K$ is $\mathbb{Z}[\alpha]$. The ring $\mathcal{O}_K$ has class number one, and has a unit group of rank one with the non-torsion part generated by the fundamental unit $\varepsilon_0 = 1-4\alpha+2\alpha^2$. The only ramified rational primes in $\mathcal{O}_K$ are 3 and 5, and they are both totally ramified. We write $\mathfrak{p}_3=2-\alpha,\mathfrak{p}_5=\alpha$ for the primes of $\mathcal{O}_K$ above 3 and 5, respectively. Note that $3 = \mathfrak{p}_3^3 \varepsilon_0^{-1}$.

We need to determine the common factors of $\alpha x - y$ and $\alpha^2 x^2+\alpha x y + y^2$ in $\mathcal{O}_K$. Let $\pi$ be a prime of $\mathcal{O}_K$ that divides both terms. Then $y \equiv \alpha x \pmod{\pi}$, and $0\equiv \alpha^2 x^2+\alpha x y + y^2 \equiv 3\alpha^2 x^2 \pmod{\pi}$. We cannot have $\pi \mid x$ since $x$ and $y$ are coprime, so $\pi = \mathfrak{p}_3$ or $\pi = \mathfrak{p}_5$. 

Neither $\alpha x - y$ nor $\alpha^2 x^2+\alpha x y + y^2$ can be divisible by $\mathfrak{p}_5$. Otherwise, the equation  $5x^3-y^3=(\alpha x - y)(\alpha^2 x^2+\alpha x y + y^2)$ implies $\alpha \mid y$, and therefore $5 \mid y$. If we write $y=5y_1$, then we have $5x^3-125y_1^3=3w^3$, and therefore $5 \mid w$. Writing $w = 5w_1$, we deduce $x^3-25y_1^3=75w_1^3$, and therefore $5 \mid x$, but this contradicts the fact that $x$ and $y$ are coprime.

Therefore, if $\alpha x - y$ and $\alpha^2 x^2+\alpha x y + y^2$ have a prime of $\mathcal{O}_K$ in common, it must be $\mathfrak{p}_3$. At least one of the terms is divisible by $\mathfrak{p}_3$ since $5x^3-y^3=3w^3$, and we claim $\mathfrak{p}_3 \mid (\alpha x - y)$ if and only if $\mathfrak{p}_3 \mid (\alpha^2 x^2+\alpha x y + y^2)$. The argument is essentially identical to the one used in studying $x^3+y^3=9z^3$; here we use the fact that the residue field of $\mathfrak{p}_3$ is isomorphic to the finite field with three elements. Thus, $\alpha x - y$ and $\alpha^2x^2+\alpha xy + y^2$ are both divisible by $\mathfrak{p}_3$.

We have shown that $\alpha x - y$ and $\alpha^2 x^2+\alpha x y + y^2$ have only the factor $\mathfrak{p}_3$ in common. We claim that $\mathfrak{p}_3$ divides $\alpha x - y$ exactly once. Write 
\begin{align*}
y = \alpha x + \mathfrak{p}_3^j f,
\end{align*}
where $j$ is a positive integer, $f \in \mathcal{O}_K$, and $\mathfrak{p}_3 \nmid f$. Substituting this expression for $y$, we have
\begin{align*}
\alpha^2 x^2 + \alpha x y + y^2 &= 3\alpha^2 x^2 + 3\alpha f x \mathfrak{p}_3^j + f^2 \mathfrak{p}_3^{2j} = \alpha^2 x^2 \varepsilon_0^{-1}\mathfrak{p}_3^3  + \alpha x f \varepsilon_0^{-1} \mathfrak{p}_3^{j+3} + f^2 \mathfrak{p}_3^{2j}.
\end{align*}
If we assume by way of contradiction that $j\geq 2$, then
\begin{align*}
\alpha^2 x^2 + \alpha x y + y^2 &= \mathfrak{p}_3^3 (\alpha^2 x^2 \varepsilon_0^{-1} + \alpha x f \varepsilon_0^{-1} \mathfrak{p}_3^{j}+f^2\mathfrak{p}_3^{2j-3}).
\end{align*}
Note that $\alpha^2 x^2 \varepsilon_0^{-1} + \alpha x f \varepsilon_0^{-1} \mathfrak{p}_3^{j}+f^2\mathfrak{p}_3^{2j-3}$ is not divisible by $\mathfrak{p}_3$, since $\alpha^2 x^2 \varepsilon_0^{-1}$ is not divisible by $\mathfrak{p}_3$. Hence $\alpha^2 x^2 + \alpha x y + y^2$ is exactly divisible by $\mathfrak{p}_3^3$, and we see by \eqref{eq:factored 5x3-y3} that $j$ is a multiple of 3. It follows that $9 \mid (5x^3-y^3)$, so $y^3 \equiv 5x^3 \pmod{9}$. The equation $x^3+y^3=9z^3$ then implies $0 \equiv x^3+y^3 \equiv 6x^3 \pmod{9}$, and therefore $3 \mid x$, which is a contradiction. Hence, $\alpha x - y$ is only divisible by $\mathfrak{p}_3$ to the first power. 

By considering \eqref{eq:factored 5x3-y3} again, we see by unique factorization that
\begin{align*}
\alpha x - y = (2-\alpha)r v^3,
\end{align*}
for some unit $r \in \mathcal{O}_K^\times$ and $v \in \mathcal{O}_K$. We have $r = \pm \varepsilon_0^b$, and by folding $-1$ and powers of $\varepsilon_0$ into $v$, we may assume $r = \varepsilon_0^b$ for $b \in \{0,1,2\}$. Therefore
\begin{align}\label{eq:factored ax-y}
\alpha x - y = (2-\alpha)\varepsilon_0^b v^3.
\end{align}

We claim that $b=0$. If $b=1$, write $v = f + g\alpha + h \alpha^2$, then expand out in \eqref{eq:factored ax-y} and compare the coefficients of $\alpha$ to obtain
\begin{align*}
x = &-9f^3 - 24f^2g + 120fg^2 - 45g^3 + 120f^2h \\ 
&- 270fgh - 120g^2h - 120fh^2 + 600gh^2 - 225h^3.
\end{align*}
Since every coefficient of the cubic form is divisible by 3, we have $3 \mid x$. This is a contradiction. If $b=2$, we similarly expand out and find that $3 \mid y$, which is also a contradiction. We conclude that $b=0$\footnote{Sage code to verify that $b=0$ is provided in \text{Proof-of-Prop-4d1-b-is-zero.sage} in our GitHub repository \cite{Pratt_cubes_2025}.}.

We multiply together $\alpha x - y$ and $x^2-xy+y^2$. By \eqref{eq:fac diff of cubes to get three times cubes} and \eqref{eq:factored ax-y} with $b=0$, we find
\begin{align*}
(\alpha x - y)(x^2-xy+y^2) = 3(2-\alpha)(vn)^3 = 3(2-\alpha)u^3,
\end{align*}
for some $u\in \mathcal{O}_K$. By Lemma \ref{lem:EC over K}, we must have $(x,y) = \pm (1,2)$, and therefore $(z,w) = \pm (1,-1)$.
\end{proof}

\section*{Acknowledgments}

The author is partially supported by the National Science Foundation (DMS-2418328) and the Simons Foundation (MPS-TSM-00007959).

\bibliographystyle{plain}
\bibliography{refs}

\end{document}